\documentclass[pdflatex,sn-mathphys-num]{sn-jnl}

\usepackage{multirow}%
\usepackage{amsmath,amssymb,amsfonts}%
\usepackage{amsthm}%
\usepackage{mathrsfs}%
\usepackage[title]{appendix}%
\usepackage{xcolor}%
\usepackage{textcomp}%
\usepackage{manyfoot}%
\usepackage{booktabs}%
\usepackage{algorithm}%
\usepackage{algorithmicx}%
\usepackage{algpseudocode}%
\usepackage{listings}
\usepackage{tikz}
\usepackage{graphicx} 
\usepackage{caption}
\usepackage{natbib} %

\theoremstyle{thmstyleone}%
\newtheorem{theorem}{Theorem}[section]
\newtheorem{remark}{Remark}%
\newtheorem{lemma}[theorem]{Lemma}
\newtheorem{corollary}[theorem]{Corollary}
\begin{document}

\title[Sharp Pre-Schwarzian Norm Bounds for Ma-Minda Starlike Classes]{Sharp Pre-Schwarzian Norm Bounds for Ma-Minda Starlike Classes}

\author*[1]{\fnm{Ming} \sur{Li}}\email{mingli@csust.edu.cn}

\author[1]{\fnm{Mei} \sur{Luo}}\email{rosemeil@163.com}
\equalcont{These authors contributed equally to this work.}

\affil*[1]{\orgdiv{School of Mathematics and Statistics}, \orgname{Changsha University of Science and Technology}, \orgaddress{\street{960, 2nd, Wanjiali RD(S)}, \city{Changsha}, \postcode{410114}, \state{Hunan}, \country{China}}}

\abstract{In this paper, we develop a unified framework to evaluate the pre-Schwarzian norm for the Ma–Minda starlike class. We present a direct, general computational approach. As an application, we streamline and consolidate the results from Ali and Pal (Monatsh. Math., 2023), who obtained sharp estimates for the pre-Schwarzian norm of the Janowski starlike class. Furthermore, we utilize the proposed framework to derive explicit norm formulas for both classical and newly introduced subclasses of starlike functions.}



\maketitle

\section{Introduction}\label{sec1}

Let $\mathscr{H}$ denote the class of all analytic functions in the open unit disk $\mathbb{D} := \{z \in \mathbb{C} : |z| < 1\}$. The Schwarz class $\mathscr{W}$ consists of all $\omega \in \mathscr{H}$ satisfying $\omega(0) = 0$ and $|\omega(z)| < 1$ for $z \in \mathbb{D}$. Let $\mathscr{A}$ denote the subclass of $\mathscr{H}$ comprising functions normalized by $f(0) = 0$ and $f'(0) = 1$, and let $\mathscr{S}$ be the subclass of $\mathscr{A}$ of univalent functions in $\mathbb{D}$. A function $f \in \mathscr{A}$ is called \textit{starlike} if $f(\mathbb{D})$ is starlike with respect to the origin, i.e., $[0, w_0] \subset f(\mathbb{D})$ for all $w_0 \in f(\mathbb{D})$. We denote by $\mathscr{S}^*$ the class of all starlike functions. It is well-known that a function \(f \in \mathscr{S}^*\) if and only if \(\operatorname{Re}\left(zf'(z)/f(z)\right) > 0\) for \(z \in \mathbb{D} \). A function $f$ belongs to the class of strongly starlike functions of order $\alpha$ (denoted by  $\mathscr{S}\mathscr{S}(\alpha)$) if
\begin{align*}
	\left|\arg\left(\frac{z f'(z)}{f(z)}\right)\right| < \frac{\pi}{2}\alpha, \quad z \in \mathbb{D},
\end{align*}
where $0 < \alpha \leq 1$. In addition, for $0 \leq \beta < 1$, let $\mathscr{S}\mathscr{S}^*(\alpha, \beta)$ denote the generalized subclass of strongly starlike functions satisfying
\begin{align*}
	\left| \arg \left( z \frac{f'(z)}{f(z)} - \beta \right) \right| < \frac{\pi \alpha}{2}, \quad z \in \mathbb{D}.
\end{align*}

Subordination is a fundamental tool for constructing generalized function classes. An analytic function $f \colon \mathbb{D} \to \mathbb{C}$ is said to be subordinate to $g \in \mathscr{H}$, denoted by $f \prec g$, if there exists a function $\omega \in \mathscr{W}$ such that $f = g \circ \omega$.

Let $\mathcal{M}$ denote the class of analytic univalent functions $\varphi$ in $\mathbb{D}$ with positive real part, satisfying $\varphi(0) = 1$, $\varphi'(0) > 0$, and whose image $\varphi(\mathbb{D})$ is symmetric with respect to the real axis and starlike with respect to $1$. Based on the concept of subordination, Ma and Minda \cite{1992Ma} unified the starlike class as
\[
\mathscr{S}^*(\varphi) = \left\{ f \in \mathscr{A} : \frac{zf'(z)}{f(z)} \prec \varphi(z) \right\},
\]
where $\varphi \in \mathcal{M}$. The class $\mathscr{S}^*(\varphi)$ is known as the Ma-Minda class of starlike functions. Obviously, \( \mathscr{S}^*((1+z/1-z))=\mathscr{S}^* \), \( \mathscr{S}^*(((1 + z)/(1 - z))^\alpha)=\mathscr{S}\mathscr{S}^*(\alpha) \) and \( \mathscr{S}^*(((1 + (1 - 2\beta)z)/(1 - z))^\alpha)=\mathscr{S}\mathscr{S}^*(\alpha,\beta) \).
Taking $\varphi(z) = (1 + (1 - 2\alpha)z)/(1 - z)$ with $0 \leq \alpha < 1$ yields the class of starlike functions of order $\alpha$, denoted by $\mathscr{S}^*(\alpha)$. For parameters $-1 \leq B < A \leq 1$, the selection $\varphi(z) = (1 + Az)/(1 + Bz)$ recovers the well-known Janowski starlike class $\mathscr{S}^*(A, B)$ \cite{1970Janowski}. For more examples, we refer the reader to \cite{Lecko2025,2019Masih,2024Masih,Sun2025}.

For an analytic function $f \in \mathscr{A}$ with $f'(z)\neq 0$ in $\mathbb{D}$, the pre-Schwarzian derivative and pre-Schwarzian norm are defined by
\begin{align*}
	P_f(z) := \frac{f''(z)}{f'(z)} \quad \text{and} \quad \|P_f\| := \sup_{z \in \mathbb{D}} (1 - |z|^2) |P_f(z)|,
\end{align*}
respectively. Yamashita \cite{1976Yamashita} proved in 1976 that $\|P_f\|$ is finite if and only if $f$ is uniformly locally univalent in $\mathbb{D}$. Furthermore, if $\|P_f\| < 2$, then $f$ is bounded in $\mathbb{D}$ (Kim and Sugawa \cite{2002Kim}). For univalent functions $f \in \mathscr{S}$, the classical estimate $\|P_f\| \leq 6$ holds, and conversely, the condition $\|P_f\| \leq 1$ implies univalence of $f$ in $\mathbb{D}$. Both bounds $6$ and $1$ are best possible (see Becker \cite{1972Becker} and Becker and Pommerenke \cite{1984Becker}). 
Sugawa \cite{1998Sugawa} first obtained the sharp pre-Schwarzian norm bound for strongly starlike functions of order $\alpha$. Yamashita \cite{1999Yamashita} established the sharp estimates $\|P_f\| \leq 6 - 4\alpha$ for $f \in \mathscr{S}^*(\alpha)$. In recent years, attention has been paid to more general classes. Ali and Pal \cite{Ali2023p} studied pre-Schwarzian norms for functions in $\mathscr{S}^*(A,B)$, while Wu and Zhu \cite{Wu2025} considered the corresponding norms for functions in $\mathscr{S}\mathscr{S}^*(\alpha, \beta)$.
For more related work, we refer the reader to \cite{2006Kim,2024Wang,Allu2026,Ahamed2026}.

In this paper, we investigate the pre-Schwarzian norm for functions in the Ma-Minda starlike class. In Section 2, we develop a unified and direct method to derive sharp pre-Schwarzian norm bounds for functions in $\mathscr{S}^*(\varphi)$, establishing a general theoretical framework applicable to various starlike subclasses. In Section 3, we focus on the Janowski starlike class $\mathscr{S}^*(A,B)$ and refine the results of Ali and Pal \cite{Ali2023p}. In Section 4, by applying the general framework developed in Section 2, we derive explicit pre-Schwarzian norm expressions for several classical and new subclasses of Ma-Minda starlike functions.

\section{Pre-Schwarzian norm of the Ma-Minda class}\label{sec2}
To establish sharp estimates for the pre-Schwarzian norm of functions belonging to the Ma-Minda class of starlike functions, we first recall the Dieudonn\'e's lemma.

\begin{lemma}[Dieudonn\'e's lemma]\cite{1983Duren}\label{lem:D}
	Let $z_0$ and $w_0$ be given points in $\mathbb{D}$ with $z_0 \neq 0$. Then for all $\omega \in \mathscr{W}$ satisfying $\omega(z_0) = w_0$, the region of values of $\omega'(z_0)$ is the closed disk
	\begin{align*}
		\left| \zeta - \frac{w_0}{z_0} \right| \leq \frac{|z_0|^2 - |w_0|^2}{|z_0|(1 - |z_0|^2)}.
	\end{align*}
	Moreover, if $\omega'(z_0)$ lies on the boundary of this disk, then $\omega$ has the form
	\begin{align}\label{eq:extremal}
		\omega(z) = z \cdot \frac{\lambda\left(\dfrac{z - z_0}{1 - \bar{z}_0 z}\right) + \dfrac{w_0}{z_0}}{1 + \lambda\dfrac{\bar{w}_0}{\bar{z}_0}\left(\dfrac{z - z_0}{1 - \bar{z}_0 z}\right)}
	\end{align}
	for some constant $\lambda$ with $|\lambda| = 1$. In particular, we obtain the sharp inequality
	\[
	|\omega'(z_0)| \leq \left| \frac{w_0}{z_0} \right| + \frac{|z_0|^2 - |w_0|^2}{|z_0|(1 - |z_0|^2)},
	\]
	with equality holds if and only if $\lambda = w_0 |z_0|^2 / \left(|w_0| z_0^2\right)$.  Setting $w_0 = z_0$ in \eqref{eq:extremal} leads to $\omega(z) \equiv z$.
\end{lemma}
For $0 < s \leq t < 1$, define
\begin{align*}
	K(s,t) = \frac{s}{t} + \frac{t^2 - s^2}{t(1 - t^2)}.
\end{align*}
Note that $\lim_{(s,t)\to(0,0)} K(s,t) = 1$ for $0 \leq s \leq t < 1$, so we set $K(0,0) = 1$. Kim and Sugawa \cite{2006Kim} established the following inequality:
\begin{align}\label{lem:kimsugawa}
	K(s,t) \leq \frac{1 - s^2}{1 - t^2}, \quad 0 \leq s \leq t < 1.
\end{align}

Next, we formulate the sharp norm estimate for the class \(\mathscr{S}^*(\varphi)\). The following result is inspired by the methodological framework employed in \cite{2011Kanas}, where Kanas and Sugawa derived sharp bounds for the Schwarzian norm of the class of close-to-convex functions. 
\begin{theorem}\label{thm:main}
	For \(f\in \mathscr{S}^*(\varphi)\), the sharp inequality
	\[||P_f||\leq \sup_{0<s<t<1} F(s,t)\]
	holds, where
	\begin{align}\label{Fpq}
		F(s,t)=(1-s^2) \sup_{|z|=s}\left|\frac{\varphi'(w)}{\varphi(w)}\right|+\frac{(1-t^2)}{t}\sup_{|w|=s}\left|\varphi(w)-1\right|.
	\end{align}
\end{theorem}
\begin{proof}
	For \(f\in \mathscr{S}^*(\varphi)\), we have
	\begin{align*}
		\frac{zf'(z)}{f(z)}\prec\varphi(z).
	\end{align*}
	From the definition of subordination, there exists a function \(\omega\in\mathscr{W}\) such that
	\begin{align*}
		\frac{zf'(z)}{f(z)}=\varphi(\omega(z)).
	\end{align*}
	Differentiating both sides logarithmically yields
	\begin{align*}
		\frac{1}{z}+\frac{f''(z)}{f'(z)}-\frac{f'(z)}{f(z)}=\frac{\varphi'(\omega(z))\omega'(z)}{\varphi(\omega(z))}.
	\end{align*}
	Accordingly, the pre-Schwarzian derivative of $f$ is given by
	\begin{align*}
		P_f(z)=\frac{\varphi'(\omega(z))\omega'(z)}{\varphi(\omega(z))}+\frac{\varphi(\omega(z))-1}{z}.
	\end{align*}
	For any fixed $z\in\mathbb{D}$ and $w=\omega(z)$, an application of Lemma \ref{lem:D} gives 
	\begin{align*}
		(1-|z|^2)|P_f(z)| &\leq\frac{|\varphi'(w)|}{|\varphi(w)|}\cdot\left(\frac{|z|^2-|w|^2}{|z|}+\frac{(1-|z|^2)|w|}{|z|}\right)+\frac{(1-|z|^2)\left(|\varphi(w)-1|\right)}{|z|}\\
		&=(1-|z|^2)\left(\frac{|\varphi'(w)|}{|\varphi(w)|}\cdot K(|w|,|z|)+\frac{|\varphi(w)-1|}{|z|} \right) .
	\end{align*}
	Letting $t:=|z|$ and $s:=|w|$ and utilizing inequality \eqref{lem:kimsugawa}, we deduce
	\[
	(1 - |z|^2)|P_f(z)| \leq (1 - s^2)\sup_{|w|=s} \frac{|\varphi'(w)|}{|\varphi(w)|} + \frac{(1 - t^2)}{t} \sup_{|w|=s} |\varphi(w) - 1|.
	\]
	The equality holds when $\omega(z)=z$. Consequently, we obtain
	\[
	\|P_f\| \leq \sup_{0 < s < t < 1} \left( (1 - s^2)\sup_{|w|=s} \frac{|\varphi'(w)|}{|\varphi(w)|} + \frac{1 - t^2}{t} \sup_{|w|=s} |\varphi(w) - 1| \right).
	\]
	We complete the proof.
\end{proof}

For $\varphi\in\mathcal{M}$, we define
\begin{align}\label{PQ}
	P(z):=\frac{\varphi'(z)}{\varphi(z)},\quad 
	Q(z):=\varphi(z)-1,
\end{align}
which will be used throughout this paper. 
Using the above notations, we may rewrite equation \eqref{Fpq} as
\begin{align}\label{Fst}
	F(s,t)&=(1-s^2)\sup_{|w|=s}\left| P(w)\right|+\frac{1-t^2}{t}\sup_{|w|=s}\left|Q(w)\right|.
\end{align}
From the definition of $\varphi(z)$, $P(z)$ and $Q(z)$ are readily analytic in $\mathbb{D}$. Accordingly, we represent these functions by their Maclaurin  series expansions
\[
P(z) = \sum_{n=1}^{\infty}\gamma_n z^n \quad \text{and} \quad Q(z) = \sum_{n=1}^{\infty}p_n z^n.
\]
If $\gamma_n \geq 0$ and $p_n \geq 0$ for all integers $n \geq 1$, then we have
\begin{align*}
	\sup_{|w|=s}|P(w)|=\sum_{n=1}^{\infty}\gamma_n s^n=P(s)\quad \text{and}\quad
	\sup_{|w|=s}|Q(w)|=\sum_{k=1}^{\infty}p_n s^{n}=Q(s).
\end{align*}
Thus, $F(s,t)$ in \eqref{Fst} can be rewritten as
\begin{align}\label{nonnegativeFst}
	F(s,t)
	&=(1-s^2)P(s)+\frac{1-t^2}{t}Q(s)=(1-s^2)\frac{\varphi'(s)}{\varphi(s)}+\frac{(1-t^2)\big(\varphi(s)-1\big)}{t}.
\end{align}
Differentiating $F(s,t)$ in \eqref{nonnegativeFst} with respect to $t$, 
we obtain
\begin{align*}
	\frac{\partial F(s,t)}{\partial t}&=-\frac{\left(t^2+1\right) (\varphi (s)-1)}{t^2}=-\frac{t^2+1}{t^2}Q(s),
\end{align*}
which is non-positive for $0<s<t<1$. Thus $F(s,t)$ is non-increasing with respect to $t$ on $[s,1]$. Accordingly, $F(s,t)\leq F(s,s)$. We then present the following result.
\begin{theorem}\label{Thm:nonnegativephi}
	Let $\varphi \in \mathcal{M}$. Suppose that the functions $P(z)$ and $Q(z)$ defined in \eqref{PQ} possess non-negative Maclaurin coefficients. 
	Then for every function $f\in \mathscr{S}^*(\varphi)$,
	\begin{align*}
		\|P_f\|\leq \sup_{0<s<1}\frac{1-s^2}{s} \left( \frac{s\varphi'(s)}{\varphi(s)} + \varphi(s) - 1 \right).
	\end{align*}
\end{theorem}

On the other hand, if $P(z)$ and $-Q(z)$  take non-positive odd-order Maclaurin coefficients and non-negative even-order Maclaurin coefficients, then \eqref{Fst} can be rewritten as 
\begin{align*}
	F(s,t)&=(1-s^2)P(-s)+\frac{1-t^2}{t}(-Q(-s))\\
	&=(1-s^2)\frac{\varphi'(-s)}{\varphi(-s)}-\frac{(1-t^2)(\varphi(-s)-1)}{t}.
\end{align*}
Adopting the same analytical strategy as in the non-negative coefficient case. Since
\begin{align*}
	\frac{\partial F(s,t)}{\partial t}&=\frac{\left(t^2+1\right) (\varphi (-s)-1)}{t^2}=\frac{t^2+1}{t^2}Q(-s)\leq 0
\end{align*}
for $0<s<t<1$, it immediately follows that $F(s,t)\leq F(s,s)$. Let $\phi(s)=\varphi(-s)$, we establish the following result.
\begin{theorem}\label{Thm:oddandeven}
	Let $\varphi \in \mathcal{M}$. Suppose that $P(z)=\varphi'(z)/\varphi(z)$ and $1-\varphi(z)$ admit non-positive odd Taylor coefficients and non-negative even Taylor coefficients at $z=0$. If the function $\phi(z)=\varphi(-z)$ 
	, then for any $f\in \mathscr{S}^*(\varphi)$, the inequality
	\begin{align*}
		\|P_f\|\leq \sup_{0<s<1}\frac{1-s^2}{s}\left(-\frac{s\phi'(s)}{\phi(s)}-\phi(-s)+1\right)
	\end{align*}
	holds.
\end{theorem}

\section{Pre-Schwarzian norm of the Janowski starlike class}\label{sec4}
We begin with some auxiliary results.
\begin{lemma}\label{lem:hab}
	For the function defined by
	\begin{align}\label{hab}
		h_{AB}(s)=A^2 B s^4+s^3 (2 A^2+2 A B)+s^2(A^2 B+5 A+2 B)+s (4 A B+4)+A+2 B,
	\end{align}
	we have the following assertions.
	\begin{itemize}
		\item[(1)] For $-1 \leq B \leq -A < 0 < A \leq 1$, the function $h_{AB}(s)$ is strictly increasing on $(0,1)$ and admits exactly one root $s_1 \in (0,1]$. In particular, $s_1 = 1$ whenever $B = -1$.
		
		\item[(2)] For $-1 \leq B < A \leq 0$, the equation $h_{AB}(s) = 0$ has exactly one root $s_1 \in (0,1]$. Furthermore, $h_{AB}(s) < 0$ for all $s \in (0,s_1)$ and $h_{AB}(s) > 0$ for all $s \in (s_1,1)$. In particular, $s_1 = 1$ whenever $B = -1$.
		
		\item[(3)] For $-A \leq B \leq -A/2 < 0 < A \leq 1$, we have $h_{AB}(s) < 0$ for all $s \in (-1,0)$. Moreover, the equality $h_{AB}(s) = 0$ holds if and only if either $s = 0$ with $A + 2B = 0$ or $s = -1$ with $A = 1$.
		
		\item[(4)] For $-A/2 \leq B < A <1$, the equation $h_{AB}(s) = 0$ has exactly one zero $s_2 \in (-1,0]$. In addition, $h_{AB}(s) < 0$ for $s \in (-1,s_2)$ and $h_{AB}(s) > 0$ for $s \in (s_2,0)$. In particular, $s_2 = 0$ whenever $A + 2B = 0$. For $A = 1$, we have $h_{AB}(-1)=0$. If $1/3\leq B<1$,  $h_{AB}(s) >0$ for all $s\in(-1,0]$. If $-1\leq B<1/3$, the equation $h_{AB}(s) = 0$ has a unique root $s_2\in(-1,0]$.
	\end{itemize}
\end{lemma}
\begin{proof}
	It is easy to check that
	\begin{align}\label{h01}
		h_{AB}(0)=A+2B,\qquad \text{and }\qquad h_{AB}(1)=2 \left(A^2+3 A+2\right) (B+1).
	\end{align}
	
	(1) For $-1 \leq B\leq-A<0< A \leq 1$, one can readily verify that $h_{AB}(0)< 0$ and $h_{AB}(1) \geq 0$, which implies that $h_{AB}(s)$ admits at least one root in $(0,1]$. For $0 < A < 1$ and $s \in (0,1)$, we further establish that
	\begin{align*}
		h'_{AB}(s)&=B \left(A^2 \left(4 s^3+2 s\right)+2 A \left(3 s^2+2\right)+4 s\right)+6 A^2 s^2+10 A s+4\\
		&\geq2 (1-s) \left(2 A^2 s^2+\left(3 A-A^2\right) s-2 A+2\right)\\
		&=2(1-s)\left( 2(1-A)+As(3-A)+2A^2s^2\right)\\
		&\geq 4(1-s)\left( (1-A)+As+A^2s^2\right)> 0.
	\end{align*}
	Consequently, $h_{AB}(s)$ is  strictly increasing on $(0,1)$ and has exactly one zero in $(0,1]$. In the special case $B = -1$, this zero is attained at $s_1 = 1$. 
	
	(2) For $-1 \leq B < A \leq 0$, by \eqref{h01}, we observe that $h_{AB}(0) < 0$ and $h_{AB}(1) \geq 0$. Differentiating $h_{AB}(s)$ with respect to $s$, we obtain
	\begin{align*}
		h_{AB}''(s)&=B \left(12 \left(A s+\frac{1}{2}\right)^2+2 A^2+1\right)+12 A^2 s+10 A\\
		&< A \left(12 \left(A s+\frac{1}{2}\right)^2+2 A^2+1\right)+12 A^2 s+10 A\\
		&=12 A^3 s^2+24 A^2 s+2 \left(A^2+7\right) A\\
		&\leq 2 A \left(7 A^2+12 A+7\right)
		\leq 0,
	\end{align*}
	for all $s\in [0,1]$. This implies that $h_{AB}(s)$ is concave on the interval $[0,1]$.
	
	When $B = -1$, we have $h_{AB}(1) = 0$ and $h'_{AB}(1) = 2(3A^2 + 5A + 2)(B + 1) = 0$. It follows that $h'_{AB}(s) > 0$ for all $s \in (0,1)$, which implies that $h_{AB}(s)$ is strictly increasing on $(0,1)$. Consequently, $s = 1$ is the unique solution to the equation $h_{AB}(s) = 0$ in $[0,1]$.
	
	When $B \neq -1$ and $A \geq -2/3$, we similarly observe that $h'_{AB}(1) \geq 0$, from which it follows that $h'_{AB}(s) > 0$ for all $s \in [0,1)$. As a result, the equation $h_{AB}(s) = 0$ has exactly one root in the interval $(0,1)$. 
	
	When $B \neq -1$ and $-1 < A < -2/3$, we have $h'_{AB}(1) < 0$ and $h'_{AB}(0) = 4AB + 4 > 0$. By the Intermediate Value Theorem, $h'_{AB}(s)$ admits at least one root in the interval $(0,1)$. Since $h_{AB}(0) < 0$ and $h_{AB}(1) > 0$, we deduce that $h_{AB}(s)$ crosses the $s$-axis, increases to a maximum, and then decreases to $h_{AB}(1)$ without intersecting the $s$-axis a second time. Consequently, the equation $h_{AB}(s) = 0$ has exactly one root in $(0,1)$.

	(3) When $-A \leq B \leq -A/2 < 0 < A \leq 1$, we have $h_{AB}(0) = A + 2B \leq 0$ and $h_{AB}(-1) = 2(A^2 - 3A + 2)(B - 1) \leq 0$. The function $h_{AB}(s)$ can be rewritten as
	\begin{align*}
		h_{AB}(s) &= B\left(A^2\left(s^4 + s^2\right) + 2As\left(s^2 + 2\right) + 2s^2 + 2\right) + 2A^2s^3 + 5As^2 + A + 4s.
	\end{align*}
	For $s \in (-1,0)$, we first note that $-s(s^2 + 2)/(s^4 + s^2) > {3}/{2}$. This implies that
	\[
	A^2\left(s^4 + s^2\right) + 2As\left(s^2 + 2\right) + 2s^2 + 2 \geq (s+1)^2\left(s^2 + 2\right) > 0
	\]
	for all $A \in [0,1]$. Consequently, we get
	\begin{align*}
		h_{AB}(s) &\leq -\frac{A}{2}\left(A^2\left(s^4 + s^2\right) + 2As\left(s^2 + 2\right) + 2s^2 + 2\right) + 2A^2s^3 + 5As^2 + A + 4s \\
		&= -\frac{1}{2}s\left(A^3\left(s^3 + s\right) - 2A^2\left(s^2 - 2\right) - 8As - 8\right).
	\end{align*}
	Moreover, for all $A \in [0,1]$ and $s \in (-1,0)$, we have
	\begin{align*}
		\frac{\partial}{\partial A}\left[A^3\left(s^3 + s\right) - 2A^2\left(s^2 - 2\right) - 8As - 8\right] = 3A^2\left(s^3 + s\right) - 4A\left(s^2 - 2\right) - 8s \geq 0.
	\end{align*}
	This implies that
	\[
	A^3\left(s^3 + s\right) - 2A^2\left(s^2 - 2\right) - 8As - 8 \leq -(s-1)^2(s+4) \leq 0.
	\]
	Consequently, $h_{AB}(s) < 0$ for all $s \in (-1,0)$, and the equation $h_{AB}(s) = 0$ holds if and only if $s = 0$ when $A + 2B = 0$ or $s = -1$ when $A = 1$.
	
	(4) For $-A/2 \leq B < A \leq 1$, it follows that
	\begin{align*}
		h_{AB}(0) = A + 2B \geq 0 \qquad \text{and} \qquad h_{AB}(-1) = 2(A^2 - 3A + 2)(B - 1) \leq 0.
	\end{align*}
	By the Intermediate Value Theorem, $h_{AB}(s)$ admits at least one root over the interval $[-1,0]$. Further straightforward calculations yield
	\begin{align}
		h_{AB}(s) &= A^2 s^2 \left(B s^2 + B + 2s\right) + 2AB \left(s^2 + 2\right)s + 5As^2 + A + 2\left(B s^2 + B + 2s\right),\nonumber \\
		h_{AB}'(s) &= A^2 s^2 (2Bs + 2) + 2A^2 s \left(B s^2 + B + 2s\right) + 4AB s^2 + 2AB \left(s^2 + 2\right) + 10As + 2(2Bs + 2),\nonumber \\
		h_{AB}''(s) &= 2A^2 B s^2 + 2A^2 \left(B s^2 + B + 2s\right) + 4A^2 s (2Bs + 1) + 12AB s + 10A + 4B,\nonumber \\
		h_{AB}'''(s) &= 12AB(As + 1) + 12A^2(Bs + 1) \geq 0. \label{h3}
	\end{align}
	It is readily verified that $h_{AB}''(0) = 2\left(A^2 B + 5A + 2B\right) > 0$ and $h_{AB}''(-1) = 2\left(7A^2 - 6A + 2\right)B + 2\left(5A - 6A^2\right)$. Consequently, 
	for $0< A\leq \left(\sqrt{65}-3\right)/7$ or  $A\geq \left(\sqrt{65}-3\right)/7$ with $B\geq(6 A^2-5 A)/(7 A^2-6 A+2)$, we have
	\[
	h_{AB}''(-1) \geq 0.
	\]
	Therefore, by using \eqref{h3}, we get $h_{AB}''(s) \geq 0$ for all $s \in [-1,0]$, which implies that $h_{AB}(s)$ is concave on the interval $[-1,0]$.
	
	If $A\neq 1$, since $h_{AB}(-1) < 0$ and $h_{AB}(0) \geq 0$, by using the convexity, we deduce that $h_{AB}(s)$ has exactly one root in $(-1,0]$ (denoted by $s_2$) for the both cases. In particular, if $A + 2B = 0$, then $s_2 = 0$. 
	
	If $A=1$ with $B\geq 1/3$, we have $h_{AB}(-1)=0$ and $h_{AB}'(-1) = 2(3A^2 - 5A + 2)(1 - B)=0$. With the application of convexity, we deduce that  $h_{AB}'(s)>0$ for $s\in(-1,0)$ and $h_{AB}(s)>0$ for $s\in(-1,0)$.

	When $A> \left(\sqrt{65}-3\right)/7$ and $B<(6 A^2-5 A)/(7 A^2-6 A+2)$, then $h_{AB}''(-1) < 0$. Combining this inequality with $h_{AB}''(0) \geq 0$, there exists a point $s_* \in (-1,0]$ such that $h_{AB}''(s_*) = 0$. As a result, 
	$h_{AB}'(s)$ is decreasing on $(-1,s_*]$ and increasing on $[s_*,0]$.
	Note that $h_{AB}'(-1) = 2(3A^2 - 5A + 2)(1 - B) \leq 0$ and $h_{AB}'(0) = 4AB + 4 > 0$. Therefore, the equation $h_{AB}'(s) = 0$ possesses exactly one root $s_{\#} \in (-1,0)$. Thus, we deduce that $h_{AB}'(s) \leq 0$ for all $s \in (-1,s_{\#})$ and $h_{AB}'(s) > 0$ for all $s \in (s_{\#},0)$. 
	
	If $A\neq 1$, given $h_{AB}(-1)<0$ and $h_{AB}(0) \geq 0$, one can further deduce that $h_{AB}(s) = 0$ has exactly one root in $(-1,0]$.
	
	If $A=1$ and $B<1/3$, we have $h_{AB}(-1) = 0$ and $h_{AB}(0) \geq 0$. Therefore, we derive that $h_{AB}(s) = 0$ admits two roots within $[-1,0]$ in this scenario. 
	
\end{proof}
\begin{lemma}\label{lem:Case3}
	For \( -A < B < 0 < A < 1 \), define
	\begin{align*}
		f(s,A) = &
		- A^2 \left( B^4 s^6 - 3B^2 \left( 5B^2 - 4 \right) s^4 + 8B \left( B^2 - 1 \right) s^3 + 3s^2 - 1 \right) \\
		&+ A \left( -B^5 s^6 - 17B^4 s^3 + 3B^3 \left( s^4 + s^2 \right) + B^2 \left( 17s^2 - 3 \right) s - 6Bs^2 + B + 3s \right) \\
		&+ 2 \left( B^2 - 1 \right) \left( 3B^2 s^2 + 1 \right)+A^3 \left( B^2 - 1 \right) s^3 (Bs - 1)^3.
	\end{align*}
	We claim that \( f(s,A) < 0 \) for all \( s \in (0, 1) \).
\end{lemma}
\begin{proof}
	Taking the derivative of $f(s,A)$ with respect to $s$, we have
	\begin{align*}
		\frac{1}{3}\frac{\partial f(s,A)}{\partial s} &= 3 4B^2 ( B^2 - 1 ) s 
		- A \left( 2B^5 s^5 + 17B^4 s^2 - 2B^3 ( 2s^3 + s ) + B^2 ( 1 - 17s^2 ) + 4Bs - 1 \right) \\ 
		&- 2A^2 s ( B^4 s^2 ( s^2 - 10 ) + 4B^3 s + 8B^2 s^2 - 4Bs + 1 )+A^3 ( B^2 - 1 ) s^2 (Bs - 1)^2 (2Bs - 1)  \\
		&=: g(s, A).
	\end{align*}
	We claim that \( g(s, A) > 0 \) whenever \( -A < B < 0 < A < 1 \) and \( s \in (0, 1) \).  
	Accordingly, for any fixed $A\in(0,1)$, \( f(s,A) \) is increasing on \((0, 1) \), which implies $f(s,A)\leq f(1,A)$ for all $s\in(0,1)$. By direct computation, we obtain
	\begin{align*}
		f(1,A) &= \left( B^2 - 1 \right) \Big( A^3 (B - 1)^3 + 2A^2 \left( 7B^2 - 4B + 1 \right) - A \left( B^3 + 17B^2 - 5B + 3 \right) + 6B^2 + 2 \Big).
	\end{align*}
	Differentiating $f(1,A)$ with respect to $A$, we get 
	\begin{align*}
		\frac{\partial}{\partial A}f(1,A)=(-1 + B^2)\left(   3 A^2 (-1 + B)^3 + 
		4 A (1 - 4 B + 7 B^2)-3 + + 5 B - 17 B^2 - B^3\right) .
	\end{align*}
	The second bracket on the right-hand side  is a quadratic function in $A$ and is negative for $B\in(-A,0)$. Since $(B-1)^3<0$, and the discriminant satisfies
	\begin{align*}
		\Delta=&4 (1 + B)^2 (-5 + 20 B - 38 B^2 + 36 B^3 + 3 B^4)<0,
	\end{align*}
	we conclude that $f(1,A)$ is an increasing function of $A$ on $(0,1]$. Consequently, which yields $f(1,A)\leq f(1,1)=0$ holds for all $A\in(0,1]$. We get the conclusion.
	
	We now revisit the foregoing statement. Differentiating $g(s,A)$ twice with respect to $A$, we obtain
	\[
	\frac{\partial^2 g(s, A)}{\partial A^2} 
	= 6A ( B^2 - 1 ) s^2 (Bs - 1)^2 (2Bs - 1) 
	- 4s \left( B^4 s^2 ( s^2 - 10 ) + 4B^3 s + 8B^2 s^2 - 4Bs + 1 \right),
	\]	
	One can readily verify that
	\[
	\left( B^2 - 1 \right) s^2 (Bs - 1)^2 (2Bs - 1) > 0
	\]	
	for \( B \in (-1, 0) \) and \( s \in (0, 1) \). Thus, for fixed $A,B$, \(\partial^2 g(s,A)/\partial A^2\) is strictly increasing with respect to $A$.  
	For $B \in (-1, 0)$, there exists some $s \in (0, 1)$ such that
	\[
	A_0 = A_0(B, s) = \frac{2\bigl(B^4s^4 - 10B^4s^2 + 4B^3s + 8B^2s^2 - 4Bs + 1\bigr)}{3(B^2 - 1)s(Bs - 1)^2(2Bs - 1)},
	\]
	and $\left.\partial^2 g(s,A)/\partial A^2\right|_{A=A_0} = 0$. In addition, $\partial^2 g(s,A)/\partial A^2 < 0$ for all $A \in (0, A_0)$ and $\partial^2 g(s,A)/\partial A^2 > 0$ for all $A \in (A_0, 1)$.
	
	Moreover, the partial derivative of $A_0(B,s)$ with respect to $s$ satisfies
	\begin{align*}
		\frac{\partial A_0(B, s)}{\partial s} 
		&= \frac{2\bigl( -1 + 7 B s + 2 (-8 + 3 B^2) B^2 s^2 - (-32 + 34 B^2) B^3 s^3 
			+ (-29 + 40 B^2) B^4 s^4 - 5 B^5 s^5\bigr)}{3(B^2 - 1)s^2(1 - 2Bs)^2(Bs - 1)^3}\\
		&\leq\frac{2(-1 - B s) (1 - B s)^3 (1 - 5 B s)}{3(B^2 - 1)s^2(1 - 2Bs)^2(Bs - 1)^3} \leq 0,
	\end{align*}
	since $B^2<1$. 
	Meanwhile, $A_0\!\left(0,2/{3}\right) = 1$, we define
	\begin{align*}
		A_0\!\left(B, \frac{2}{3}\right) 
		= \frac{81 - 216 B + 288 B^2 + 216 B^3 - 344 B^4}{3 (3 - 2 B)^2 (-3 + 4 B) (-1 + B^2)}.
	\end{align*}
	Differentiating $A_0\!\left(B, 2/{3}\right) $ with respect to $B$, we obtain
	\begin{align*}
		\frac{\partial}{\partial B}A_0\!\left(B, \frac{2}{3}\right) 
		&=\frac{ 2B \bigl(-1701 + 486 B + 4428 B^2 - 372 B^3 - 4528 B^4 + 336 B^5 + 
			1376 B^6\bigr)}{3 (3 - 4 B)^2 (-3 + 2 B)^3 (-1 + B^2)^2 }\\
		&=\frac{2B}{3 (3 - 4 B)^2 (-3 + 2 B)^3 (-1 + B^2)^2}
		\Bigl(  -875 B^6 + 3200 (B+1) B^5 - 7219 (B+1)^2 B^4 \\
		&\qquad + 20796 (B+1)^3 B^3 - 23517 (B+1)^4 B^2 + 10692 (1+B)^5 B 
		- 1701 (1+B)^6\Bigr).
	\end{align*}
	Since $B \in (-1,0)$, we have $\partial A_0(B,2/3)/\partial B < 0$. Thus $A_0(B,2/3)$ is strictly decreasing in $B$ on $(-1,0)$ and
	\[
	A_0\!\left(B,\frac{2}{3}\right) \geq A_0\!\left(0,\frac{2}{3}\right) = 1.
	\]
	Consequently, we conclude $A_0 > 1$ for all $s \in (0,2/3)$.
	This implies $\partial^2 g(s, A)/\partial A^2 < 0$ for all $s \in (0, 2/3)$ and $B \in (-1, 0)$.
	
	It is straightforward to verify that
	\begin{align}
		g(s, -B)&= -B(B^2 - 1)(Bs + 1)^2(2B^3s^3 - 9B^2s^2 - 1) > 0,   \label{g(-B)}\\
		g(s,1) 
		&= (1+B)(s-1)^2 \Big[ (1 - B)(1 - s)^3 + (3 - 7B - 2B^2 + 4B^3)s(1 - s)^2 \nonumber\\
		&\qquad + (3 - 11B + B^2 - B^3)s^2(1 - s) + (1 - 5B + 3B^2 - 7B^3)s^3 \Big] > 0  \label{g(1)}
	\end{align}
	for all \(B\in(-1,0)\) and \( s \in (0, 1) \).
	Thus, for any fixed $s \in (0, 2/3)$, using the concavity of $g(s,A)$, we conclude that
	$g(s, A) > 0$ for all $s \in (0, 2/3)$, $-B < A < 1$ and $B \in (-1, 0)$.
	
	Taking the derivative of $g(s,A)$ with respect to $A$, we obtain
	\begin{align*}
		\frac{\partial g(s, A)}{\partial A} &= 3A^2 (B^2 - 1) s^2 (Bs - 1)^2 (2Bs - 1)  - 4As \left( B^4 s^2 (s^2 - 10) + 4B^3 s + 8B^2 s^2 - 4Bs + 1 \right) \\
		&\quad - 2B^5 s^5 - 17B^4 s^2 + 2B^3 (2s^3 + s) + B^2 (17s^2 - 1) - 4Bs + 1 .
	\end{align*}
	Fix $s \in (0,1)$ and $B \in (-A,0)$ with $A > 0$. Since $3(B^2-1)s^2(Bs-1)^2(2Bs-1) > 0$, it follows that $\partial g(s,A)/\partial A$ is convex in $A$.
	
	Direct computation yields
	\begin{align*}
		\left.\frac{\partial g(s, A)}{\partial A}\right|_{A=1} &=
		4B^5 s^5 - B^4 \left( 4s^3 + 15s^2 - 40s + 17 \right) s^2 - 2B^3 \left( 3s^4 - 8s^2 + 8s - 1 \right) s \\
		&+ B^2 \left( 15s^4 - 32s^3 + 14s^2 - 1 \right)- 4B \left( 3s^2 - 4s + 1 \right) s + 3s^2 - 4s + 1.
	\end{align*} 
	For $t>0$ and $u>0$, we define
	\begin{align*}
		s=\frac{t+2}{t+3},   \qquad B=-\frac{u}{u+1}.
	\end{align*}
	One can readily verify that $s>2/3$ and $-1<B<0$.
	Then we get
	\begin{align*}
		\left.\frac{\partial g(s, A)}{\partial A}\right|_{A=1} =-\frac{N(t,u)}{(3 + t)^5 (1 + u)^5},
	\end{align*}
	where
	\begin{align*}
		N(t, u) = & (t+3)^3 (2t+3) + u(t+3)^2 (2t+3) (9t+23) + u^2 (t+3) (4t^4 + 92t^3 + 494t^2 + 988t + 663) \\
		& + u^3 (8t^5 + 176t^4 + 1190t^3 + 3566t^2 + 4959t + 2607)  + 2u^4 (3t+5) (8t^3 + 58t^2 + 140t + 113) \\
		& + 2u^5 (t+2) (2t+5)^2.
	\end{align*}
	Since $N(t, u) >0$ for all $t>0$ and $u>0$, we conclude that $\left.\partial g(s, A)/\partial A\right|_{A=1}  < 0$ for all $s \in (2/3, 1)$ and $B \in (-1, 0)$. 
	
	Therefore,
	For any fixed $s \in (2/3,1)$ and $B \in (-1,0)$, by exploiting the convexity of $\partial g(s, A)/\partial A$, we deduce that $\partial g(s,A)/\partial A$ has at most one zero in $(-B,1)$.
	
	Furthermore, if $\partial g(s,A)/\partial A$ has no zero in $(-B,1)$, then by convexity it follows that $\partial g(s,A)/\partial A < 0$ for all $A \in (-B,1)$. Consequently, $g(s,A)$ is strictly decreasing in $A$ on $(-B,1)$.
	In view of \eqref{g(1)}, we conclude that $g(s, A) > 0$ for all $s \in (2/3, 1)$, $A \in (-B, 1)$, and $B \in (-1, 0)$.	
	
	If $g_A(s,A)$ has a unique zero $A_1 \in (-B,1)$ for fixed $s \in (2/3,1)$ and $B \in (-1,0)$, then  $\partial g(s,A)/\partial A < 0$ for all $A\in(-B,A_1)$ and $\partial g(s,A)/\partial A > 0$ for all $A\in (A_1,1)$, that is, $g(s,A)$ is increasing in $A$ on $(-B,A_1)$ and decreasing in $A$ on $(A_1,1)$. Since $g(s,-B) > 0$, $g(s,1) > 0$, we conclude that $g(s,A) > 0$ for all $s \in (2/3,1)$, $B \in (-1,0)$, and $A \in (-B,1)$.
\end{proof}

\begin{theorem}\label{thm:AB}
	If \(f \in \mathscr{S}^*(A,B)\)
	for \(-1 \leq B < A \leq 1\), then the following inequality holds:
	\begin{align*}
		\|P_f\| \leq
		\begin{cases}
			-\dfrac{ (A-B)\left(s_1^2-1\right) (A s_1+2)}{(A s_1+1) (B s_1+1)}, & \text{for } |B|>|A|, \\
			\dfrac{(A-B)\left((s'_2)^2-1\right)  (A s'_2-2)}{(A s'_2-1) (B s'_2-1)},  & \text{for } 0<B<A\leq1, \\
			\dfrac{ (A-B)\left(1 - s_3^2\right) \left(2 + A s_3 (-1 + B s_3)\right)}{(-1 + A s_3) (-1 + B s_3) (1 + B s_3)},& \text{for }-A\leq B<0<A\leq1.
		\end{cases}
	\end{align*}
	Here, \(s_1\) and \(s_2' = -s_2\) denote respectively the roots of \(h_{AB}(s)\) (defined in \eqref{hab}) lying in \([0,1]\).
	Moreover, \(s_3\) is the unique root in \([0,1]\) of the equation
	\begin{align*}
		-\frac{2 A \left(A^2-1\right)}{(A-B) (A s-1)^3}
		+\frac{2 B \left(B^2-1\right)}{(A-B) (B s-1)^3}
		+\frac{2 \left(B^2-1\right)}{(B s+1)^3}=0.
	\end{align*}
	In particular, if \(B=-1\), then \(s_1=1\), which yields \(\|P_f\| \leq 2(A+2);\) if \(A+2B=0\), then \(s_2=0\), so \(\|P_f\| \leq 2(A-B)\).
\end{theorem}
\begin{proof}
	Let \(\varphi_{A,B}(z)=(1+Az)/(1+Bz)\) for \(-1\leq B<A\leq 1\). We have
	\begin{align*}
		P_ {A,B}(z)&=\frac{\varphi'_{A,B}(z)}{\varphi_{A,B}(z)}=\frac{A-B}{(A z+1) (B z+1)}
		=\sum_{n=0}^\infty(-1)^n(A^{n+1}-B^{n+1})z^n,\\
		Q_{A,B}(z)&=\varphi_{A,B}(z)-1=\frac{(A-B) z}{ (B z+1)}=(A-B)\sum_{n=1}^\infty(-1)^{n-1} B^{n-1}z^n.
	\end{align*}
	\textbf{Case 1:} When 
\(|B|>|A|\), i.e. \(-1\leq B<A\leq 0\) or \(-1\leq B\leq-A<0<A\leq1\), we have $(-1)^n(A^{n+1}- B^{n+1}) > 0$ and $(-1)^{n-1} B^{n-1} \geq 0$. That is, the functions $Q_{A,B}$ and $P_{A,B}$ have non-negative Maclaurin coefficients. Thus, we derive that $\displaystyle\sup_{|z|=s}|P_{A,B}(z)|=P_{A,B}(s)$ and $\displaystyle\sup_{|z|=s}|Q_{A,B}(z)|=Q_{A,B}(s)$. 
	We know from Theorem \ref{Thm:nonnegativephi} that
	\begin{align*}
		||P_f||\leq\sup_{0<s<1}\frac{1-s^2}{s}\left(\frac{s\varphi'_{A,B}(s)}{\varphi_{A,B}(s)} +\varphi_{A,B}(s)-1\right) =\sup_{0 < s < 1}\frac{(1-s^2)(A-B)(A s + 2)}{(A s+ 1)(B s + 1)}.
	\end{align*}
	Define  
	\begin{align*}
		F_1(s)=\frac{(1-s^2)(A-B)(A s + 2)}{(A s+ 1)(B s + 1)}.
	\end{align*}
	Differentiating $F_1(s)$ with respect to $s$, we obtain
	\begin{align*}
		F'_1(s)=\frac{(B-A)h_{AB}(s)}{(As+1)^2 (Bs+1)^2},
	\end{align*}
	and
	\begin{align*}
		h_{AB}(s)=A^2 B s^4 + s^3 \left(2A^2 + 2AB\right) + s^2 \left(A^2 B + 5A + 2B\right) + s (4AB + 4) + A + 2B.
	\end{align*}
	By items (1) and (2) of Lemma \ref{lem:hab}, we conclude that the equation $F'_1(s) = 0$ has exactly one root in the interval $[0,1]$, denoted by $s_1$. The function $F_1(s)$ is increasing on $(0,s_1)$ and decreasing on $(s_1,1)$. Consequently,
	\begin{align*}
		\|P_f\| \leq \frac{(1-s_1^2)(A-B)(A s_1 + 2)}{(A s_1 + 1)(B s_1 + 1)}.
	\end{align*}
	In particular, when $B = -1$, we have $s_1 = 1$, and
	\begin{align*}
		\|P_f\| \leq \sup_{0 < s < 1} \frac{(A+1)(s+1)(A s + 2)}{A s + 1} = 2(A + 2).
	\end{align*}
	\textbf{Case 2:} When $0\leq B<A\leq 1$, we have \(A^{n+1}-B^{n+1}>0\). Consequently, we derive that the functions $-Q_{A,B}$ and $P_{A,B}$ possess non-positive odd-order Maclaurin coefficients and non-negative even-order Maclaurin coefficients.Then we have 
	\begin{align}\label{-QAB}
		\sup_{|z|=s}|P_{A,B}(z)|&
		=\sum_{n=0}^\infty(-1)^n(A^{n+1}-B^{n+1})(-s)^n
		=P_{A,B}(-s),\\
		\sup_{|z|=s}|Q_{A,B}(z)|&
		=(A-B)\sum_{n=1}^\infty(-1)^n B^{n-1}(-s)^n
		=-Q_{A,B}(-s).
	\end{align} 
	Let $\phi_{A,B}(s)=\varphi_{A,B}(-s)=(1-As)/(1-Bs)$.
	Thus, by Theorem \ref{Thm:oddandeven}, we have
	\begin{align*}
		\|P_f\|\leq \sup_{0<s<1}\frac{1-s^2}{s}\left(-\frac{s\phi_{A,B}'(s)}{\phi_{A,B}(s)}-\phi_{A,B}(-s)+1\right)=\sup_{0<s<1}\frac{\left(s^2-1\right) (A-B) (A s-2)}{(A s-1) (B s-1)}.
	\end{align*}
	Define 
	\begin{align*}
		F_2(s)=\frac{\left(s^2-1\right) (A-B) (A s-2)}{(A s-1) (B s-1)}.
	\end{align*}
	Taking derivative of $F_2(s)$ with respect to $s$, we have
	\begin{align*}
		F_2'(s)=\frac{(A-B)h_{AB}(-s)}{(A s-1)^2 (B s-1)^2},\end{align*}
	where
	\begin{align*}
		h_{AB}(s)=A^2 B s^4+s^3 \left(2 A^2+2 A B\right)+s^2 \left(A^2 B+5 A+2 B\right)+s (4 A B+4)+A+2 B.
	\end{align*}
	For  $0 \leq B < A \leq 1$, item (4) of Lemma \ref{lem:hab} implies that the equation $F_2'(s) = 0$ has exactly one root in the interval $[0,1)$, denoted by $s_2' = -s_2$. The function $F_2(s)$ is increasing on $(0,s_2')$ and decreasing on $(s_2',1)$, thus
	\begin{align*}
		\|P_f\|\leq\sup_{0<s<1}\frac{\left(s^2-1\right) (A-B) (A s-2)}{(A s-1) (B s-1)}= \frac{(A-B)\left((s'_2)^2-1\right)  (A s'_2-2)}{(A s'_2-1) (B s'_2-1)}.
	\end{align*}{\bf Case 3:} When $-A \leq B  \leq 0 < A \leq 1$, we have $A^{n+1} - B^{n+1} > 0$. The function $P_{A,B}(z)$ possesses non-positive odd-order Maclaurin coefficients and non-negative even-order Maclaurin coefficients. While The function $Q_{A,B}$ have non-negative Maclaurin coefficients. Thus we get
	\begin{align*}
		&\sup_{|z|=s}|P_{A,B}(z)|
		=P_{A,B}(-s)=\frac{A-B}{(1-A s) (1-B s)},\\
\text{and}\quad  &\sup_{|z|=s}|Q_{A,B}(z)|=Q_{A,B}(s)=\frac{(A-B) s}{ B s+1}.
	\end{align*}
	Therefore, we obtain
	\begin{align*}
		||P_f||&\leq\sup_{0 < s < 1} (A - B) \left( \frac{1 - s^2}{(-1 + A s)(-1 + B s)} + \frac{s - s t^2}{t + B s t} \right).
	\end{align*}
	Define 
	\begin{align*}
		F_3(s,t)=\frac{1 - s^2}{(-1 + A s)(-1 + B s)} + \frac{s - s t^2}{t + B s t} .
	\end{align*}
	Take derivative of $F_3(s,t)$ with respect to $t$, we have
	\begin{align*}
		\frac{\partial F_3(s,t)}{\partial t}=-\frac{ s \left(1 + t^2\right)}{(1 + B s) t^2}<0,
	\end{align*}
	which implies $F(s,t)$ is decreasing for $t$ on $[s,1]$. Therefore, we can conclude that
	\begin{align*}
		||P_f||\leq\sup_{0 < s < 1}\frac{(A-B) \left(1 - s^2\right) \left(2 + A s (-1 + B s)\right)}{(-1 + A s) (-1 + B s) (1 + B s)}.
	\end{align*}
	Setting 
	\begin{align*}
		F_3(s)=\frac{ (A-B)\left(1 - s^2\right) \left(2 + A s (-1 + B s)\right)}{(-1 + A s) (-1 + B s) (1 + B s)},
	\end{align*}
	a straightforward calculation yields
	\begin{align*}
		F''_3(s)&=\frac{2(A-B)f(s,A)}{(-1 + A s)^3 (-1 + B s)^3 (1 + B s)^3},
	\end{align*}
	where 
	\begin{align*}
		f(s,A) = &
		- A^2 \left( B^4 s^6 - 3B^2 \left( 5B^2 - 4 \right) s^4 + 8B \left( B^2 - 1 \right) s^3 + 3s^2 - 1 \right) \\
		&+ A \left( -B^5 s^6 - 17B^4 s^3 + 3B^3 \left( s^4 + s^2 \right) + B^2 \left( 17s^2 - 3 \right) s - 6Bs^2 + B + 3s \right) \\
		&+ 2 \left( B^2 - 1 \right) \left( 3B^2 s^2 + 1 \right)+A^3 \left( B^2 - 1 \right) s^3 (Bs - 1)^3 ,
	\end{align*}
	By Lemma \ref{lem:Case3}, we know that $f(s,A)<0$ for all $s\in(0,1)$ and $-A \leq B \leq 0 < A \leq 1$. That is, $F_3(s)$ is a concave function on $s\in[0,1]$. A direct computation gives
	\[
	F_3'(s) = \frac{A^2-1}{(As-1)^2} - \frac{B^2-1}{(Bs-1)^2} - \frac{(A-B)(Bs^2+B+2s)}{(Bs+1)^2}.
	\]
	Consequently,
	\[
	F_3'(0) = A(A-B) > 0, \qquad F_3'(1) = -\frac{2(A-B)(2+A(B-1))}{(A-1)(B-1)(B+1)} < 0,
	\]
	provided that $-A \leq B \leq 0 < A \leq 1$. Then, we can conclude that $F'_3(s)$ has only one root in $s\in (0,1)$, say $s_3$, and $F_3(s)$ is increasing on $[0,s_3]$ and decreasing on $[s_3,1]$. Hence, it follows that
	\begin{align*}
		||P_f||\leq\frac{ (A-B)\left(1 - s_3^2\right) \left(2 + A s_3 (-1 + B s_3)\right)}{(-1 + A s_3) (-1 + B s_3) (1 + B s_3)}.
	\end{align*} 
	We complete the proof.
\end{proof}
\begin{remark}
	The first and second cases recover the results of Ali and Pal \cite{Ali2023p}, which were obtained by a different method. In their work, the extremal function was shown to be
	\begin{align*}
		K_{A,B}(z) =
		\begin{cases}
			z e^{Az}, & \text{if } B = 0, \\[4pt]
			z(1 + Bz)^{\frac{A}{B} - 1}, & \text{if } B \neq 0.
		\end{cases}
	\end{align*}
\end{remark}

\section{Applications}
By applying Theorems \ref{Thm:nonnegativephi}-\ref{Thm:oddandeven}, we examine pre-Schwarzian norms over various subclasses of Ma–Minda starlike functions.

\begin{corollary} \cite{1998Sugawa}
	For \(f \in \mathscr{S}\mathscr{S}^*(\alpha)\) with \(\alpha\in(0,1]\), we have
	\begin{align*}
		\|P_f\|\leq 2\alpha+\frac{4\alpha x_0^{\alpha+1}}{x_0^2+1},
	\end{align*}
	where \(x_0\in(1,+\infty)\) is the unique root of the equation
	\begin{align*}
		(1-\alpha)x^{\alpha+2}+(1+\alpha)x^\alpha-x^2-1=0.
	\end{align*}
\end{corollary}

\begin{proof}
	Let \(\varphi(z)=\left((1+z)/(1-z)\right)^\alpha\). By  simple calculations, we have
	\begin{align*}
		P(z)&=\frac{\varphi'(z)}{\varphi(z)}=2\alpha\sum_{n=0}^\infty z^{2n},\\
		Q(z)&=\varphi(z)-1=\sum_{n=1}^{\infty}\sum_{k=0}^n \binom{\alpha}{k} \binom{\alpha-k+n-1}{n-k}z^n.
	\end{align*}
	Obviously, the functions \(P(z)\) and \(Q(z)\) have nonnegative Maclaurin coefficients for \(0<\alpha\leq 1\). 
	By applying Theorem \ref{Thm:nonnegativephi}, we obtain the inequality
	\begin{align*}
		\|P_f\|&\leq\sup_{0<t<1}(1-t^2)\left(\frac{\varphi'(t)}{\varphi(t)}+\frac{\varphi(t)-1}{t}\right)=2\alpha+\sup_{0<t<1}\frac{(1-t^2)\left(-1+\left(\frac{1+t}{1-t}\right)^\alpha\right)}{t}.
	\end{align*}
	Following the approach of Sugawa \cite{1998Sugawa}, we deduce
	\begin{align*}
		\|P_f\| \leq 2\alpha+\frac{4\alpha x_0^{\alpha+1}}{x_0^2+1},
	\end{align*}
	where \(x_0\in(1,+\infty)\) is the unique root of the equation
	\begin{align*}
		(1-\alpha)x^{\alpha+2}+(1+\alpha)x^\alpha-x^2-1=0.
   \end{align*}
\end{proof}
\begin{corollary}
	For \(0<\lambda<{\pi}/{2}\) and \(f\in S^*(e^{\lambda z})\), we have
	\begin{align*}
		\|P_f\|\leq \frac{(1-t_0^2)(e^{\lambda t_0}+\lambda t_0-1)}{t_0},
	\end{align*}
	where \(t_0\) is the unique root of the equation
	\begin{align*}
		1+r^2-2\lambda r^3-e^{\lambda r}(1-\lambda r+r^2+\lambda r^3)=0
	\end{align*}
in (0,1).
\end{corollary}
\begin{proof}
	Let \(\varphi_{\lambda }(z)=e^{\lambda z}, 0<\lambda<{\pi}/{2}\). Through straightforward calculations, we obtain
	\begin{align*}
		P_\lambda(z)=\frac{\varphi'_\lambda (z)}{\varphi_\lambda (z)}=\lambda,\quad \text{and}\quad
		Q_\lambda(z)=\varphi_\lambda (z)-1=\sum_{n=1}^\infty\frac{ \lambda^n  z^n}{n!}.
	\end{align*}
	Since $0 < \lambda < \pi/2$,  both $P_\lambda$ and $Q_\lambda$ have non-negative Maclaurin coefficients. 
	 Applying Theorem \ref{Thm:nonnegativephi}, we obtain
	\[
	\|P_f\| \leq \sup_{0<t<1} (1-t^2)\left(\frac{\varphi_\lambda'(t)}{\varphi_\lambda(t)} + \frac{\varphi_\lambda(t)-1}{t}\right) = \sup_{0<t<1} \frac{(1-t^2)(e^{\lambda t} + \lambda t - 1)}{t}.
	\]
	Define 
	\begin{align*}
		h_{\lambda}(t)=\frac{(1-t^2)(e^{\lambda t}+\lambda t-1)}{t}.
	\end{align*}
	Taking the derivative of $h_{\lambda}(t)$ with respect to t, we have
	\begin{align*}
		t^2h_{\lambda}'(t)=1+t^2-2\lambda t^3-e^{\lambda t}(1-\lambda t+t^2+\lambda t^3)
	\end{align*}
	and
	\begin{align*}
		t^3h_{\lambda}''(t)
		&=-2(1+\lambda t^3)+e^{\lambda t}(2-2\lambda t +\lambda^2 t^2-2\lambda t^3-\lambda^2 t^4)=:k(t).
	\end{align*}
	We have
	\begin{align*}
		k'(t)=-\lambda t^2(6+e^{\lambda t}(6+6\lambda t-\lambda^2+t^2\lambda^2)).
	\end{align*}
	It is readily verified that \(k'(t) < 0\) for \(t \in (0,1)\) and \(\lambda \in (0,\pi/2)\), which implies that \(k(t)\) is strictly decreasing on the interval \((0,1)\). Since \(k(0) = 0\), we can conclude that \(k(t) < 0\) for all \(t \in (0,1)\). Therefore, \(h''(t) < 0\) holds for every \(t \in (0,1)\), meaning that \(h'(t)\) is also strictly decreasing on \((0,1)\). In particular, we have $\lim_{t\to 0} h_{\lambda}'(t) = \lambda^2 > 0$ and $h_{\lambda}'(1) = -2(e^\lambda - 1) - 2\lambda < 0$. By the intermediate value theorem and the strict monotonicity of $h_{\lambda}'$, there exists a unique $t_0 \in (0,1)$ such that $h_{\lambda}'(t_0) = 0$.
	Hence $h_{\lambda}(t)$ is increasing on $(0,t_0)$ and decreasing on $(t_0,1)$, so that $h_{\lambda}$ attains its maximum at $t_0$. Consequently,
	\begin{align*}
		\|P_f\|\leq \frac{(1-t_0^2)(e^{\lambda t_0}+\lambda t_0-1)}{t_0}.
	\end{align*}
\end{proof}
\begin{corollary}
	For \(f\in S^*_{con}\), i.e., \(zf'(z)/f(z)\prec 3/(3+(\alpha-3)z-\alpha z^2),\) \(-3<\alpha\leq0 \), we have \[||P_f||\leq 2(6+\alpha)/(3+\alpha).\]
\end{corollary}
\begin{proof}
	Let \(\varphi_{con}= 3/(3+(\alpha-3)z-\alpha z^2)\), and \(-3<\alpha \leq0\). By simple calculations, we obtain
	\begin{align*}
		P_{con}(z)&=\frac{\varphi'_{con}(z)}{\varphi_{con}(z)}=\left(\frac{-1}{z-1}+\frac{-\alpha}{3+\alpha z} \right)=\sum_{n=0}^\infty\left[1+(-1)^{n+1}\frac{\alpha^{n+1}}{3^{n+1}} \right] z^n, \\
		Q_{con}(z)&=\varphi_{con}(z)-1=\sum_{n=1}^\infty\left[ \frac{3}{\alpha+3}+\frac{(-1)^n\alpha^{n+1}}{(\alpha+3)3}\right] z^n.
	\end{align*}
	It follows that \(P_{\mathrm{con}}(z)\) and \(Q_{\mathrm{con}}(z)\) possess nonnegative Maclaurin coefficients. Therefore, by applying Theorem  \ref{Thm:nonnegativephi}, we obtain
	\begin{align*}
		||P_f||&\leq\sup_{0<t<1}  (1-t^2)\left(\frac{\varphi'_{con}(t)}{\varphi_{con}(t)}+\frac{\varphi_{con}(t)-1}{t}\right)
		=\sup_{0<t<1}\frac{(1+t)\left(6+\alpha (-2+3 t) \right) }{3+\alpha t}.
	\end{align*}
	Define  
	\begin{align*}
		h_{con}(t)=\frac{(1+t)\left(6+\alpha (-2+3 t) \right) }{3+\alpha t}.
	\end{align*}
	Taking the derivative of $h_{con}(t)$ with respect to $t$, we obtain
	\[
	h_{con}'(t)=\frac{18 + 3\alpha (-1 + 6t) + \alpha^2 (2 + 3t^2)}{(3 + \alpha t)^2}.
	\]
	The numerator of $h_{con}'(t)$ is a quadratic function in $\alpha$ and is nonnegative for $\alpha \in (-3, 0]$. Since $2 + 3t^2 > 0$ and the discriminant satisfies $\Delta = 27(-5 +4t(t-1)) < 0$, it follows that $h_{con}(t)$ is increasing on $(0,1)$. Consequently,
	\[
	\|P_f\| \leq h_{con}(1) = \frac{2(6+\alpha)}{3+\alpha}.
	\]
\end{proof}
\begin{corollary}
	For \(f\in S^*_{cs},\), i.e., \(zf'(z)/f(z)\prec 1+z/((1-z)(1+a z)), 0< a\leq 1/2\), we have \(||P_f||\leq{2(2+a)}/{(1+a)}.\) The estimate is sharp.
\end{corollary}
\begin{proof}
	Let \(\varphi_{cs}=1+z/((1-z)(1+a z))\) and \(0< a\leq 1/2\). Through direct calculations, we derive the following results:
	\begin{align*}
		P_{cs}(z)&=\frac{\varphi'_{cs}(z)}{\varphi_{cs}(z)}=\frac{1}{1-z}-\frac{a}{1+az}+\frac{a(2z-1)}{az^2-az-1}=\sum_{n=0}^\infty\left[ 1-(-1)^na^{n+1}- \frac{1}{r_1^{n+1}}-\frac{1}{r_2^{n+1}} \right] z^n,\\
		Q_{cs}(z)&=\varphi_{cs}(z)-1=\frac{1}{a+1}\sum_{n=0}^\infty[1+(-1)^na^{n+1}]z^{n+1},
	\end{align*}
	where 
	\(
	r_{1,2} ={a\pm\sqrt{a^2+4a}}/(2a).
	\)
	Since $0< a\leq 1/2$, we deduce that \(1/r_1 \in [0, 1/2]\), \(1/r_2 \in [-1, 0]\). By applying the fact \(r_1 r_2 = -{1}/{a}\), we have
	\begin{align*}
		1 - (-1)^n a^{n+1} - \frac{1}{r_1^{n+1}} - \frac{1}{r_2^{n+1}} = \left(1 - \frac{1}{r_1^{n+1}}\right)\left(1 - \frac{1}{r_2^{n+1}}\right) \geq 0.
	\end{align*} 
	It follows that the functions \(P_{\text{cs}}(z)\) and \(Q_{\text{cs}}(z)\) have non-negative Maclaurin coefficients for \(0 < a \leq {1}/{2}\).
	Therefore,  by Theorem \ref{Thm:nonnegativephi}, we obtain
	\begin{align*}
		||P_f||&\leq\sup_{0 < t < 1}  (1-t^2)\left(\frac{\varphi_{cs}'(t)}{\varphi_{cs}(t)}+\frac{\varphi_{cs}(t)-1}{t}\right)=\sup_{0<t<1}-\frac{(2+at)(1+t)}{(1+at)(at^2-at-1)}.
	\end{align*}
	Set 
	\begin{align*}
		h_{cs}(t)=-\frac{(2+at)(1+t)}{(1+at)(at^2-at-1)}.
	\end{align*}
	For \(0<t<1\) and \(0< a\leq 1/2\), we have
	\begin{align*}
		h'_{cs}(t) = \frac{2 - 3a + (6a - 4a^2)t + (2a + 7a^2 - a^3)t^2 + (4a^2 + 2a^3)t^3 + a^3 t^4}{(1 + at)^2 \left(1 - a t(1 - t)\right)^2} \geq 0.
	\end{align*}
	Thus, \(h_{cs}(t)\) is non-decreasing function on \((0,1)\), and
	\begin{align*}
		\|P_f\| \leq h_{cs}(1) = \frac{2(2 + a)}{1 + a}.
	\end{align*}
	To demonstrate  its sharpness, we consider the function \(f_1\) defined by
	\begin{align*}
		f_1(z) = z \exp\left( \int_{0}^{z}\frac{1}{(1-t)(1+at)} \,dt \right).
	\end{align*}
	Direct computation yields
	\begin{align*}
		\frac{f_1''(z)}{f_1'(z)} = \frac{2+az}{(z-1)(1+az)\left(1 - a z(1 - z)\right)}.
	\end{align*}
	The pre-Schwarzian norm of \(f_1\) is given by
	\begin{align*}
		||P_{f_1}||=\sup_{z\in\mathbb{D}}(1-|z|^2)|P_{f_1}|=\sup_{z\in\mathbb{D}}(1-|z|^2)\left|\frac{2+az}{(z-1)(1+az)(-1+a(-1+z)z)} \right| .
	\end{align*}
	When $z=r>0$, we obtain
	\begin{align*}
		\sup_{0\leq r<1}(1-r^2)\left|\frac{2+ar}{(r-1)(1+ar)\left( -1+a(-1+r)r\right) } \right|
		=\sup_{0\leq r<1}-\frac{(2+ar)(1+r)}{(1+ar)(ar^2-ar-1)}=\frac{2(2+a)}{1+a}.
	\end{align*}
	This completes the proof.
\end{proof}
\begin{remark}
	When \(a=0\), this result also holds.
\end{remark}

\begin{corollary}
	If \(f\in S^*_{lim},\), i.e., \(zf'(z)/f(z)\prec (1+z)(1-s^*z), 0\leq s^*\leq 1/3\), then we have
	\begin{align*}
		||P_f||\leq \frac{(1-s^*+2s^*t_{s^*})(1+t_{s^*})}{(1+s^*t_{s^*})}+(1-t_{s^*}^2)(1-s^*+s^*t_{s^*}),
	\end{align*}
	where \(t_{s^*}\) is the unique root of the equation
	\begin{align*}
		\frac{-3s^*t^4+(2s^{*3}-8s^{*2})t^3+(s^{*3}+6s^{*2}-7s^{*})t^2+(2s^{*2}+6s^*-2)t+s^{*2}+s^*+1}{(1+s^*t)^2}=0
	\end{align*}
in \((0,1)\).
\end{corollary}
\begin{proof}
	Let \(\varphi_{s^*}(z)=(1+z)(1-s^*z), s^*\in [0,1/3]\); then  we have
	\begin{align*}
		P_{s^*}(z)&=\frac{\varphi'_{s^*}(z)}{\varphi_{s^*}(z)}=\frac{1}{1+z}+\frac{s^*}{s^*z-1}=\sum_{n=0}^\infty\left[(-1)^n-s^{*(n+1)} \right]z^n, \\
		Q_{s^*}(z)&=1-\varphi_{s^*}(z)=z(s^*-1) +s^*z^2.
	\end{align*}
	Since \(0\leq s^*\leq 1/3\), the Maclaurin series of \(P_{s^*}(z)\) and \(Q_{s^*}(z)\) are alternating. Let \(\phi_{s^*}(s)=\varphi_{s^*}(-s)\). By Theorem \ref{Thm:oddandeven}, we have
	\begin{align*}
		\|P_f\|&\leq\sup_{0<t<1}\frac{1-t^2}{t}\left(-\frac{t\phi'_{s^*}(t)}{\phi_{s^*}(t)}-\phi_{s
			^*}(-t)+1\right)\\
&=\sup_{0 < t < 1}\left( \frac{(1+t)(1+s^*(2t-1))}{1+ts^*}+(1-t^2)(1-s^*+s^*t)\right) .
	\end{align*}
	Define 
	\begin{align*}
		h_{s^*}(t)=\frac{(1+t)(1+s^*(2t-1))}{1+ts^*}+(1-t^2)(1-s^*+s^*t).
	\end{align*}
	Taking the derivative of $h_{s^*}(t)$ with respect to $t$, we obtain
	\begin{align*}
		h_{s^*}'(t)&=\frac{-3s^*t^4+(2s^{*3}-8s^{*2})t^3+(s^{*3}+6s^{*2}-7s^{*})t^2+(2s^{*2}+6s^*-2)t+s^{*2}+s^*+1}{(1+s^*t)^2},\\
		h_{s^*}''(t)&=\frac{-2\left( 1 - 2 s^* + s^{*3} + (6 s^* - 3 s^{*2}) t + (12s^{*2} - 3 s^{*3}) t^2 + (10 s^{*3} - s^{*4}) t^3 + 3 s^{*4}t^4\right) }{(1 + s^* t)^3}.
	\end{align*}
	It is straightforward to see that $h_{s^*}''(t) < 0$. Hence, $h_{s^*}'(t)$ is strictly decreasing on $[0,1]$. Moreover, for all \(0\leq s^*\leq 1/3\), we have
	\[
	\lim_{t\to 0^+} h_{s^*}'(t) = 1 + s^* + s^{*2} > 0,\qquad \text{and}\quad
	\lim_{t\to 1^-} h_{s^*}'(t) = \frac{-1 + s^*}{1 + s^*} < 0.
	\]
	 Therefore, we conclude that $h_{s^*}'(t)$ has a unique root $t_{s^*}\in(0,1)$, and $h_{s^*}(t)$ attains its maximum at $t_{s^*}$. Consequently,
	\begin{align*}
		\|P_f\| \leq h_{s^*}(t_{s^*}) = \frac{(1-s^*+2s^*t_{s^*})(1+t_{s^*})}{(1+s^*t_{s^*})} + (1-t_{s^*}^2)(1-s^*+s^*t_{s^*}).
	\end{align*}
	This completes the proof.
\end{proof}

\begin{corollary}\cite{Allu2026}
	If \(f\in S\left( (1-z)^{-\alpha}\right) )\) with \(0<\alpha\leq 1\), then
	\begin{align*}
		\|P_f\|\leq\begin{cases}
			(1-t_1^2)\left(\dfrac{\alpha}{1-t_1}+\dfrac{-1+(1-t_1)^{-\alpha}}{t_1} \right), &\text{for } 0<\alpha<1,\\
			4, &\text{for } \alpha=1,
		\end{cases}
	\end{align*}
	where \(t_1\) is the unique root in \((0,1)\) of the equation
	\begin{align*}
		(1-t)^{-\alpha}\left(\alpha t\bigl(1+t+(1-t)^\alpha t\bigr)+\bigl(-1+(1-t)^\alpha\bigr)(1+t^2) \right) =0.
	\end{align*}
\end{corollary}

\begin{proof}
	Let \(\varphi_\alpha(z)=(1-z)^{-\alpha}\) for \(0<\alpha\leq 1\). Then we have
	\begin{align*}
		P_\alpha(z)&=\frac{\varphi'_\alpha(z)}{\varphi_\alpha(z)}=\alpha\sum_{n=0}^\infty z^n,\\
		Q_\alpha(z)&=\varphi_\alpha(z)-1=\sum_{n=1}^{\infty} \binom{\alpha+n-1}{n} z^n.
	\end{align*}
	The functions \(P_\alpha(z)\) and \(Q_\alpha(z)\) have non-negative Maclaurin coefficients for \(0<\alpha\leq 1\).
	Thus, by Theorem \ref{Thm:nonnegativephi}, we get
	\begin{align*}
		\|P_f\| &\leq \sup_{0 < t < 1} (1-t^2)\left(\frac{\varphi_\alpha'(t)}{\varphi_\alpha(t)}+\frac{\varphi_\alpha(t)-1}{t} \right) = \sup_{0 < t < 1} (1-t^2)\left(\frac{\alpha}{1-t}+\frac{-1+(1-t)^{-\alpha}}{t} \right).
	\end{align*}
	When \(\alpha=1\), we obtain
	\[
	\|P_f\|\leq 2(1+t)\leq 4.
	\]
	For the case \(0<\alpha<1\). We define
	\begin{align*}
		h_{\alpha}(t)=(1-t^2)\left(\frac{\alpha}{1-t}+\frac{-1+(1-t)^{-\alpha}}{t} \right).
	\end{align*}
	Taking the derivative of \(h_{\alpha}(t)\) with respect to \(t\), we obtain
	\begin{align*}
		h_{\alpha}'(t) &=\frac{\alpha t\bigl(1+t+(1-t)^\alpha t\bigr)+\bigl(-1+(1-t)^\alpha\bigr)(1+t^2)}{t^2(1-t)^{\alpha}}, \\
		t^3h_{\alpha}''(t)&= -\frac{2\bigl(-1 + (1-t)^\alpha\bigr) + (2 + 2\alpha - 2(1-t)^\alpha)t - \alpha(1+\alpha)t^2 - (\alpha - 1)\alpha t^3}{(1-t)^{\alpha+1} } =: k_1(t).
	\end{align*}
	For \(0<\alpha<1\) and \(0<t<1\), we have
	\begin{align*}
		k'_1(t) &=\frac{\alpha (\alpha-1) t^2\bigl(4+\alpha+\alpha t-2 t\bigr)}{(1-t)^{2+\alpha}} < 0.
	\end{align*}
	That is, \(k_1(t)\) is strictly decreasing on \(t \in (0,1)\). It follows that \(h_1(t) \leq h_1(0) = 0\), hence \(h_{\alpha}''(t) \leq 0\) for \(0 < t < 1\). Consequently, \(h_{\alpha}'(t)\) is non-increasing on \((0,1)\). Computing 
	\[
	\lim_{t \to 0^+} h_{\alpha}'(t) = \frac{\alpha(3+\alpha)}{2} > 0, \qquad
	\lim_{t \to 1^-} h_{\alpha}'(t) = -\infty,
	\]
	we conclude that \(h_{\alpha}'(t)\) has a unique root \(t_1 \in (0,1)\), and \(h_{\alpha}(t)\) attains its maximum at \(t_1\). Therefore,
	\[
	\|P_f\| \leq (1 - t_1^2) \left( \frac{\alpha}{1 - t_1} + \frac{-1 + (1 - t_1)^{-\alpha}}{t_1} \right).
	\]
\end{proof}

\section*{Statements and Declarations:}
	\vskip.05in
	\noindent{\bf Funding: }
	The present investigation was supported by the \textit{Natural Science Foundation of Hunan Province} under Grant no. 2026JJ60325.
	\vskip.05in
	\noindent{\bf Financial interests:} 
 The authors declare they have no financial interests.
 	\vskip.05in
	\noindent{\bf Author contributions:} All authors contributed to the study conception and design. Material preparation, data collection and analysis were performed by Ming Li and Mei Luo. The first draft of the manuscript was written by Mei Luo and all authors commented on previous versions of the manuscript. All authors read and approved the final manuscript.
\bibliography{reference}
\end{document}